\documentclass[12pt,twoside]{amsart}
\usepackage{amsthm,amsfonts,amssymb,amscd,euscript}

\usepackage[matrix,arrow]{xy}

\author{Florin Ambro}
\address{Institute of Mathematics ``Simion Stoilow'' of the Romanian
Academy\\
P.O. BOX 1-764, RO-014700 Bucharest\\
Romania.}
\email{florin.ambro@imar.ro}
\author{Mugurel Barc\u{a}u}
\address{Institute of Mathematics ``Simion Stoilow'' of the Romanian
Academy\\
P.O. BOX 1-764, RO-014700 Bucharest\\
Romania.}
\email{mugurel.barcau@imar.ro}

\newcommand{\Q}{{\mathbb Q}}
\newcommand{\Z}{{\mathbb Z}}

\newcommand{\bP}{{\mathbb P}} 

\newcommand{\lcm}{\operatorname{lcm}}

\theoremstyle{plain}
\newtheorem{thm}{Theorem}[section]

\newtheorem{lem}[thm]{Lemma}

\newtheorem{prop}[thm]{Proposition}

\theoremstyle{definition}

\newtheorem{rem}[thm]{Remark}

\theoremstyle{remark}

\begin{document}

\bibliographystyle{amsalpha+}
\title{On representations by Egyptian fractions}

\dedicatory{To Lucian B\u{a}descu on the occasion of his seventieth birthday}

\begin{abstract}
We bound the entries of the representations of a rational number as a sum of Egyptian fractions.
\end{abstract}

\maketitle


\footnotetext[1]{2010 Mathematics Subject Classification: Primary 11D68. Secondary 14E30.}
\footnotetext[2]{Keywords: Egyptian fractions, log canonical models.}


\section*{Introduction}


Let $(X,B)$ be a log canonical model with standard coefficients. That is $(X,B)$ is a log variety
with log canonical singularities, $K_X+B$ is $\Q$-ample, and the coefficients of $B$
belong to the standard set $\{1-\frac{1}{m};m\in \Z_{\ge 1}\}\cup\{1\}$. The volume of $(X,B)$ is
$v=\sqrt[d]{(K_X+B)^d}$, where $d=\dim X$. By~\cite{Ale94,Kol94,HMX12,HMX14}, the
volume $v$ belongs to a DCC set, and there exists a positive integer $r$, bounded above only
in terms of $d$ and $v$, such that the linear system $|r(K_X+B)|$ is base point free (in particular,
$r(K_X+B)$ is a Cartier divisor).
The DCC property means that if $t$ is a real number and $v>t$, then $v\ge t+\epsilon$, where
$\epsilon$ depends only on $d$ and $t$.

In this note we estimate the gap and index bounds $\epsilon$ and $r$ in the simplest possible case,
when $X$ is a projective space and the components of $B$ are hyperplanes in general position.
According to~\cite{Kol94}, the sharp bounds of the simplest case are possibly optimal
in the general case.

To formulate our main result, we define a sequence of integers $(u_{p,q})_{p,q\ge 1}$
by the recursion $u_{1,q}=q$, $u_{p+1,q}=u_{p,q}(u_{p,q}+1)$.
Then $u_{p,q}$ is a polynomial in $q$ with leading term $q^{2^{p-1}}$, and the following formulas hold:
$$
\sum_{i=1}^p\frac{1}{1+u_{i,q}}=\frac{1}{q}-\frac{1}{u_{p+1,q}}, \
\prod_{i=1}^p(1+u_{i,q})=\frac{u_{p+1,q}}{q}.
$$
The sequence $(1 + u_{p,1})_{p\ge 1}=(2,3,7,43, . . .)$ is called the Sylvester sequence in the literature
(see~\cite{Kel21,Kol94}), and also the sequence $t_{p,q} = 1 + u_{p,q}$ was considered in~\cite{LZ91}.

\begin{thm}\label{mr}
Let $(\bP^d,\sum_i b_i E_i)$ be a log structure such that the $(E_i)_i$ are general hyperplanes
and the coefficients $b_i$ belong to the standard set. Let $v=\deg(K+B)$. Let $t \ge 0$ be a
rational number, with $qt\in \Z$ for some integer $q\ge 1$.
\begin{itemize}
\item[a)] If $v>t$, then $v\ge t+\frac{q(1-\{t\})}{u_{\lfloor t\rfloor+d+3,q}}$.
\item[b)] If $v=t$, then there exists an integer $1\le r\le \frac{u_{\lfloor t\rfloor+d+2,q}}{q(1-\{t\})}$
such that the linear system $|r(K+B)|$ is base point free.
\end{itemize}
\end{thm}

Theorem~\ref{mr} is in fact combinatorial, about bounding the representations of a given rational number
as a sum of Egyptian fractions. Any positive rational number $x$ admits a representation as a sum of Egyptian fractions
$$
x=\frac{1}{m_1}+\cdots+\frac{1}{m_k},
$$
where $m_i$ are positive integers and $k$ is sufficiently large.
If $x=\frac{p}{q}$ is the reduced form, we can write $x=\sum_{i=1}^p\frac{1}{q}$.
From a representation with $k$ terms we can construct another one with $k+1$ terms, using the formula
$$
\frac{1}{m}=\frac{1}{m+1}+\frac{1}{m(m+1)}.
$$

A canonical representation is provided by the {\em greedy algorithm}: if $x>0$, let
$m\ge 1$ be the smallest integer such that $mx\ge 1$, and replace $x$ by $x-\frac{1}{m}$;
if $x=0$, stop. After each step, the numerator of the reduced fraction decreases strictly,
and therefore the algorithm stops in finite time, and produces a representation of
$x$ as a sum of $k$ Egyptian fractions ($k\le \lfloor x\rfloor+q\{x\}$ if $qx\in \Z$).

If $k$ is fixed, it is easy to see that $x$ admits only finitely many representations with $k$ Egyptian fractions.
The following is an effective version of this fact, which is a restatement of Theorem~\ref{mr}.

\begin{thm}\label{main}
Let $1\le m_1\le \cdots \le m_k$ be integers. Let $\delta\ge -1$ with $q\delta\in\Z$ for some integer $q\ge 1$.
\begin{itemize}
\item[a)] If $\sum_{i=1}^k\frac{1}{m_i}<k-\delta$, then $\sum_{i=1}^k\frac{1}{m_i}\le k-\delta-\frac{q(1-\{\delta\})}
{u_{\lfloor \delta\rfloor+2,q}}$.
\item[b)] If $\sum_{i=1}^k\frac{1}{m_i}=k-\delta$, then $\lcm(m_1,\ldots,m_k)
\le \frac{u_{\lfloor \delta\rfloor+1,q}}{q(1-\{\delta\})}$.
\end{itemize}
Moreover, equality holds in a) if and only if $\delta<0$, or
$\delta=\frac{r}{q}\in [0,1), (m_i)_i=(1,\ldots,1,\frac{1+q}{r})$, or
$1\le \delta=s-\frac{1}{q}, (m_i)_i=(1,\ldots,1,1+u_{1,q},\ldots,1+u_{s,q})$. And equality holds in b) if and
only if $\delta=s-\frac{1}{q},(m_i)_i=(1,\ldots,1,1+u_{1,q},\ldots,1+u_{s,q},u_{s+1,q})$, or
$\delta=2-\frac{r}{q}$ and $(m_i)_i=(1,\ldots,1,\frac{1+q}{r},\frac{q(1+q)}{r})$.
\end{thm}

The case $k-\delta=1$ is known (Kellogg~\cite{Kel21}, Curtiss~\cite{Cur22}, Soundararajan~\cite{Sou05}),
with b) replaced by the same bound for $m_k$ instead of the least common multiple.
We use the method of Soundararajan~\cite{Sou05}.


\section{Proof of estimates}


\begin{lem}[\cite{Sou05}]\label{sp}
Consider real numbers $x_1\ge x_2\ge \cdots \ge x_n>0$ and
$y_1\ge y_2\ge \cdots \ge y_n>0$ such that
$\prod_{i\le k} x_i\ge \prod_{i\le k} y_i$ for all $k$. Then
$\sum_i x_i\ge \sum_i y_i$, with
equality if and only if $x_i=y_i$ for all $i$.
\end{lem}

\begin{proof} Soundararajan~\cite{Sou05} deduces this lemma from Muirhead's inequality.
We give here a direct proof, by induction on $n$. If $x_i=y_i$ for some $i$, we may remove the $i$-th terms from both $n$-tuples, and
conclude by induction; therefore we may suppose $x_i\ne y_i$ for every $i$.
If $x_i > y_i$ for all $i$, the conclusion is clear.
Suppose $x_i < y_i$ for some $i$. Let $l=\min\{i;x_i<y_i\}$.
Then $l>1$ and $x_i > y_i$ for every $i<l$.
Let $t=\min\{\frac{x_{l-1}}{y_{l-1}}, \frac{y_{l}}{x_{l}}\} > 1$.
Define $(x'_i)_i$ by $x'_i = x_i$, for $i \notin \{l-1, l\}$, and $x'_{l-1} = \frac{x_{l-1}}{t}$,
$x'_l = t x_l$.
One checks that $x'_1 \geq x'_2 \geq ... \geq x'_n > 0$,
$\prod_{i=1}^k x'_i \ge \prod_{i=1}^k x_i$ for all $k$,
and $x_{l-1} + x_l > x'_{l-1} + x'_l$, hence
$
\sum_{i=1}^n x_i > \sum_{i=1}^n x'_i.
$
Since either $x'_{l-1} = y_{l-1}$ or $x'_l = y_l$, $\sum_{i=1}^n x'_i \ge \sum_{i=1}^n y_i$
by induction. Therefore $\sum_{i=1}^n x_i > \sum_{i=1}^n y_i$.
The claim on equality is clear.
\end{proof}

\begin{lem}\label{sp2}
Consider real numbers $x_1\ge x_2\ge \cdots \ge x_n>0$ and
$y_1\ge y_2\ge \cdots \ge y_n>0$ such that
$\sum_{i\ge k} x_i\ge \sum_{i\ge k} y_i$ for all $k$. Then
$\prod_i x_i\ge \prod_i y_i$, with
equality if and only if $x_i=y_i$ for all $i$.
\end{lem}

\begin{proof}
As in the previous lemma we use induction on $n$, so that we may suppose $x_i\ne y_i$ for every $i$.
In particular, $x_n>y_n$. If $x_i>y_i$ for every $i$, the claim is clear.
So suppose that $x_i<y_i$ for some $i$. Let $k=\max\{i;x_i<y_i\}$.
Then $k<n$ and $x_i>y_i$ for every $i\ge k+1$. In particular,
$$
y_{k+1}<x_{k+1}\le x_k<y_k.
$$
Define $(y'_i)_i$ by $y'_i=y_i$ for $i\notin \{k,k+1\}$,
$y'_k=y_k-\epsilon,y'_{k+1}=y_{k+1}+\epsilon$, where
$\epsilon=\min\{x_{k+1}-y_{k+1},y_k-x_k\} > 0$. The following hold:
\begin{itemize}
\item $y'_1\ge \cdots\ge y'_n>0$.
\item $\sum_{i\ge j} y_i\le \sum_{i\ge j} y'_i$, with equality for $j\ne k+1$.
And $ \sum_{i\ge j}x_i\ge \sum_{i\ge j}y'_i$ for all $j$.
\item $y'_ky'_{k+1} - y_ky_{k+1} =\epsilon(y_k-y_{k+1}-\epsilon)>0$.
Therefore $\prod_i y'_i > \prod_i y_i$.
\end{itemize}
By induction, the claim holds for $(x_i)$ and $(y'_i)$, since either
$x_k=y'_k$ or $x_{k+1}=y'_{k+1}$. Therefore $\prod_i x_i\ge \prod_i y'_i$,
so that $\prod_i x_i > \prod_i y_i$.
\end{proof}

\noindent For the next proposition we need the following lemma whose proof is obvious.

\begin{lem}\label{o1}
Let $n,p,q$ be positive integers with $1-\frac{1}{n}\le \frac{p}{q}<1$. Then $n\le q$.
\end{lem}

\begin{prop}\label{mp}
Let $s\ge 0,1\le r\le q$ be integers. If $1\le n_1 \le \cdots \le n_k$
are integers such that $\sum_{i=1}^k\frac{1}{n_i} < k-s+\frac{r}{q}$, then
$\sum_{i=1}^k\frac{1}{n_i} \leq k-s+\frac{r}{q}-\frac{r}{u_{s+1,q}}$. Equality holds
if and only if $n_i=1$ for $i\le k-s$ and $n_i=\frac{1+u_{i-k+s,q}}{r}$ for $i>k-s$.
\end{prop}

\begin{proof} We use induction on $s$ to prove that if $1\le n_1 \le \cdots \le n_k$ are integers
such that $k-s+\frac{r}{q}-\frac{r}{u_{s+1,q}}\le \sum_{i=1}^k\frac{1}{n_i}<k-s+\frac{r}{q}$,
then $n_i=1$ for $i\le k-s$ and $n_i=\frac{1+u_{i-k+s,q}}{r}$ for $i>k-s$.

If $s=0$, then $k\le \sum_{i=1}^k\frac{1}{n_i}<k+\frac{r}{q}$, so that
$n_i=1$ for all $i$.

Let $s\ge 1$. The right inequality yields $s\le k$. Denote
$m_i=1$ for $1\le i\le k-s$ and $m_i=\frac{1+u_{i-k+s,q}}{r}$ for $k-s<i\le k$.
We have
$$
\sum_{i=1}^k\frac{1}{m_i}=k-s+\frac{r}{q}-\frac{r}{u_{s+1,q}} , \
\prod_{i=1}^k m_i=\frac{u_{s+1,q}}{r^sq}.
$$
Our hypothesis can be rewritten as
$$
1-\frac{q}{u_{s+1,q}}\le \frac{q}{r}(s-k+\sum_{i=1}^k \frac{1}{n_i})<1.
$$
The middle term can be represented as a fraction with denominator $r\prod_i n_i$.
By Lemma~\ref{o1}, $\frac{u_{s+1,q}}{q}\le r\prod_{i=1}^k n_i$.
Therefore $\prod_{i=1}^k m_i\le \prod_{i=1}^k n_i$. Then we can define
$$
j=\max\{1\le l\le k; \prod_{i\ge l} m_i\le \prod_{i\ge l} n_i \}.
$$

Assume $j=k$, that is $m_k\le n_k$. Then
$\sum_{i=1}^{k-1}\frac{1}{m_i}\le \sum_{i=1}^{k-1}\frac{1}{n_i}<(k-1)-(s-1)+\frac{r}{q}$.
By induction, $n_i=m_i$ for every $i\le k-1$. It follows that $n_k=m_k$.

Assuming $j<k$, we derive a contradiction.
Then $\prod_{i\ge j}n_i\ge \prod_{i\ge j}m_i$ and
$\prod_{i\ge p}n_i< \prod_{i\ge p}m_i$ for every $j<p\le k$.
It follows that $\prod_{i=j}^p n_i>\prod_{i=j}^pm_i$ for
every $j\le p<k$. We rewrite this as
$$
\prod_{i=j}^p \frac{1}{m_i}\ge \prod_{i=j}^p\frac{1}{n_i} \ (j\le p\le k),
$$
with strict inequality for $p\ne k$. By Lemma~\ref{sp},
$\sum_{i=j}^k \frac{1}{m_i}>\sum_{i=j}^k\frac{1}{n_i}$.
On the other hand, $\sum_{i=1}^{j-1}\frac{1}{n_i}<k-s+\frac{r}{q}$.
By induction, $\sum_{i=1}^{j-1}\frac{1}{n_i}\le \sum_{i=1}^{j-1}
\frac{1}{m_i}$.
Therefore $\sum_{i=1}^k\frac{1}{n_i}<\sum_{i=1}^k\frac{1}{m_i}$,
a contradiction.
\end{proof}

\begin{rem}
Notice that since $1+u_{1,q}$ and $1+u_{2,q}$ are relatively prime, if $s\ge 2$
equality is achieved only for $r=1$.
\end{rem}

\begin{prop}\label{lcmb}
Let $s\ge 0$ and $1\le r\le q$ be integers. If $1\le n_2\le \cdots \le n_k$ are integers such that
$\sum_{i=1}^k\frac{1}{n_i}=k-s+\frac{r}{q}$, then $\lcm(n_1,\ldots,n_k)\leq \frac{u_{s,q}}{r}$.
Equality holds if and only if $n_i=1$ for $1\le i\le k-s$, $n_i=\frac{1+u_{i-k+s,q}}{r}$ for $k-s<i<k$ and $n_k=\frac{u_{s,q}}{r}$.
\end{prop}

\begin{proof} We prove by induction on $s$ that if $1\le n_2\le \cdots \le n_k$ are integers such that
$\sum_{i=1}^k\frac{1}{n_i}=k-s+\frac{r}{q}$ and $\lcm(n_1,\ldots,n_k)\ge \frac{u_{s,q}}{r}$, then
$n_i=1$ for $1\le i\le k-s$, $n_i=\frac{1+u_{i-k+s,q}}{r}$ for $k-s<i<k$ and $n_k=\frac{u_{s,q}}{r}$.

It follows that $s\ge 1$. If $s=1$, we must have $(n_i)=(1,\ldots,1,\frac{q}{r})$,
so the conclusion holds. Suppose $s\ge 2$. Let $m_i=1$ for $1\le i\le k-s$, $m_i=\frac{1+u_{i-k+s,q}}{r}$ for $k-s<i<k$ and $m_k=\frac{u_{s,q}}{r}$.
We have $m_1\le \cdots \le m_k$, $\sum_{i=1}^k\frac{1}{m_i}=k-s+\frac{r}{q}$ and
$\prod_{i=1}^km_i=\frac{u_{s,q}^2}{r^sq}$.

{\em Step 1}: We claim that $\sum_{i\ge l}\frac{1}{n_i}\ge \sum_{i\ge l}\frac{1}{m_i}$ for every $1\le l\le k$.

Indeed, equality holds for $l=1$. Let $1<l\le k-s+1$.
Then $\sum_{i<l}\frac{1}{n_i}\le l-1=\sum_{i<l} \frac{1}{m_i}$. Therefore $\sum_{i\ge l}\frac{1}{n_i}\ge
\sum_{i\ge l} \frac{1}{m_i}$.
Let $k-s+1<l\le k$. Then $\sum_{i=1}^{l-1}\frac{1}{n_i}<k-s+\frac{r}{q}=(l-1)-(l+s-k-1)+\frac{r}{q}$.
By Proposition \ref{mp}, $\sum_{i=1}^{l-1}\frac{1}{n_i}\le \sum_{i=1}^{l-1}\frac{1}{m_i}$.
Therefore $\sum_{i\ge l}\frac{1}{n_i}\ge \sum_{i\ge l} \frac{1}{m_i}$.

{\em Step 2}: By Step 1 and Lemma~\ref{sp2}, we obtain
$\prod_i\frac{1}{n_i}\ge \prod_i\frac{1}{m_i}$, that is $\prod_i n_i\le \prod_i m_i$.
And equality holds if and only if $n_i=m_i$ for all $i$.

{\em Step 3}: Denote $L=\lcm(n_1,\ldots,n_k)$. We claim that $L^2\le q\prod_{i=1}^kn_i$.

Indeed, $q\mid L$ and $n_i\mid \lcm(q,n_1,\ldots,\widehat{n_i},\ldots,n_k)$
for all $i$. Fix a prime $p$. The power of $p$ in $L$ is the highest power of $p$ occuring
in the prime decomposition of $q,n_1,\ldots,n_k$. From above, the maximum is attained at least twice.
Therefore $L^2\le q\prod_{i=1}^kn_i$.

{\em Step 4}:  We obtain $L^2\le q\prod_{i=1}^kn_i\le q\prod_{i=1}^km_i=\frac{u_{s,q}^2}{r^s}$.
Since $s\ge 2$, we obtain $L\le \frac{u_{s,q}}{r}$. We assumed the opposite inequality, so
$L=\frac{u_{s,q}}{r}$. It follows that $\prod_{i=1}^kn_i= \prod_{i=1}^km_i$, so $n_i=m_i$ for all $i$.
\end{proof}

\begin{rem}
Note that equality is achieved if $\frac{1+u_{i-k+s,q}}{r}$ for $k-s<i<k$ and $\frac{u_{s,q}}{r}$ are integers,
that is if and only if $s=1$ and $r\mid q$, or $s=2$ and $r \mid 1+q$, or $s \ge 3$ and $r = 1$.
\end{rem}

\begin{proof}[Proof of Theorem~\ref{main}]
Write $\delta=s-\frac{r}{q}$, where $s=\lfloor \delta\rfloor+1$ and $r=q(1-\{\delta\})$.
Then $k-\delta=k-s+\frac{r}{q}$, and we may apply Propositions \ref{mp} and \ref{lcmb}.
\end{proof}

\begin{proof}[Proof of Theorem~\ref{mr}]
Order the coefficients of $B$ as $0\le b_1\le \cdots\le b_k<1=b_{k+1}=\cdots=b_{k+c}$.
Let $b_i=1-\frac{1}{m_i}$, for $1\le i\le k$. Then $\sum_{i=1}^k\frac{1}{m_i}=k-(d-c+1+v)$.
Denote $r=\lcm(m_i)$. By Theorem~\ref{main}, $r \le \frac{u_{\lfloor t\rfloor+d-c+2,q}}{q(1-\{t\})}$.
Then $rB$ is a divisor with integer coefficients, and since the ambient space is $\bP^d$, the
semipositive Cartier divisor $r(K+B)$ is base point free.
\end{proof}


\end{document}